\documentclass[12pt, letterpaper]{amsart}

\usepackage[pagebackref,hypertexnames=false, colorlinks, citecolor=red, linkcolor=red]{hyperref}
\usepackage{amssymb}
\usepackage{mathrsfs}
\usepackage{color}
\usepackage{multirow}
\usepackage{appendix}

	\normalbaselines						  %

\newcommand{\R}{{\mathbb  R}}

\newcommand{\N}{{\mathbb  N}}

\newcommand{\ci}[1]{_{ {}_{\scriptstyle #1}}}

\catcode`@=11
\chardef\mathlig@atcode\count255

\def\actively#1#2{\begingroup\uccode`\~=`#2\relax\uppercase{\endgroup#1~}}
\def\mathlig@gobble{\afterassignment\mathlig@next@cmd\let\mathlig@next= }
\def\mathlig@delim{\mathlig@delim}
\def\mathlig@defcs#1{\expandafter\def\csname#1\endcsname}
\def\mathlig@let@cs#1#2{\expandafter\let\expandafter#1\csname#2\endcsname}
\def\mathlig@appendcs#1#2{\expandafter\edef\csname#1\endcsname{\csname#1\endcsname#2}}

\def\mathlig#1#2{\mathlig@checklig#1\mathlig@end\mathlig@defcs{mathlig@back@#1}{#2}\ignorespaces}

\def\mathlig@checklig#1#2\mathlig@end{%
 \expandafter\ifx\csname mathlig@forw@#1\endcsname\relax
 \expandafter\mathchardef\csname mathlig@back@#1\endcsname=\mathcode`#1%
 \mathcode`#1"8000\actively\def#1{\csname mathlig@look@#1\endcsname}%
 \mathlig@dolig#1\mathlig@delim
\fi
\mathlig@checksuffix#1#2\mathlig@end
}

\def\mathlig@checksuffix#1#2\mathlig@end{%
\ifx\mathlig@delim#2\mathlig@delim\relax\else\mathlig@checksuffix@{#1}#2\mathlig@end\fi
}
\def\mathlig@checksuffix@#1#2#3\mathlig@end{%
\expandafter\ifx\csname mathlig@forw@#1#2\endcsname\relax\mathlig@dosuffix{#1}{#2}\fi
\mathlig@checksuffix{#1#2}#3\mathlig@end
}

\def\mathlig@dosuffix#1#2{%
\mathlig@appendcs{mathlig@toks@#1}{#2}%
\mathlig@dolig{#1}{#2}\mathlig@delim
}

\def\mathlig@dolig#1#2\mathlig@delim{%
 \mathlig@defcs{mathlig@look@#1#2}{%
 \mathlig@let@cs\mathlig@next{mathlig@forw@#1#2}\futurelet\mathlig@next@tok\mathlig@next}%
 \mathlig@defcs{mathlig@forw@#1#2}{%
  \mathlig@let@cs\mathlig@next{mathlig@back@#1#2}%
  \mathlig@let@cs\checker{mathlig@chck@#1#2}%
  \mathlig@let@cs\mathligtoks{mathlig@toks@#1#2}%
  \expandafter\ifx\expandafter\mathlig@delim\mathligtoks\mathlig@delim\relax\else
  \expandafter\checker\mathligtoks\mathlig@delim\fi
  \mathlig@next
 }%
 \mathlig@defcs{mathlig@toks@#1#2}{}%
 \mathlig@defcs{mathlig@chck@#1#2}##1##2\mathlig@delim{%
  \ifx\mathlig@next@tok##1%
   \mathlig@let@cs\mathlig@next@cmd{mathlig@look@#1#2##1}\let\mathlig@next\mathlig@gobble
  \fi
  \ifx\mathlig@delim##2\mathlig@delim\relax\else
   \csname mathlig@chck@#1#2\endcsname##2\mathlig@delim
  \fi
 }%
 \ifx\mathlig@delim#2\mathlig@delim\else
  \mathlig@defcs{mathlig@back@#1#2}{\csname mathlig@back@#1\endcsname #2}%
 \fi
}%

\catcode`@\mathlig@atcode

\mathchardef\ordinarycolon\mathcode`\:
\def\vcentcolon{\mathrel{\mathop\ordinarycolon}}
\mathlig{:=}{\vcentcolon=}
\mathlig{::=}{\vcentcolon\vcentcolon=}

\numberwithin{equation}{section}

\theoremstyle{plain}
\newtheorem{theo}{Theorem}[section]

\newtheorem{lem}[theo]{Lemma}
\newtheorem{prop}[theo]{Proposition}

\theoremstyle{definition}

\theoremstyle{remark}
\newtheorem*{ex*}{Example}
\theoremstyle{remark}
\newtheorem*{exs*}{Examples}
\theoremstyle{remark}
\newtheorem*{rem*}{Remark}
\newtheorem{rem}[theo]{Remark}
\newtheorem*{rems*}{Remarks}

\title[Moment Representations]{Moment Representations of Exceptional $X_1$ Orthogonal Polynomials}
\author{Constanze~Liaw}
\address{CASPER and Department of Mathematics, Baylor University, One Bear Place \#97328,      
 Waco, TX  76798, USA}
\email{Constanze$\underline{\,\,\,}$Liaw@baylor.edu}
\urladdr{http://sites.baylor.edu/constanze$\underline{\,\,\,}$liaw/}
\thanks{This work was supported by a grant from the Simons Foundation (\#426258, Constanze Liaw).}

\author{Jessica~Stewart~Kelly}
\address{Department of Mathematics, Christopher Newport University, 1 Avenue of the Arts, Newport News, VA 23606}
\email{Jessica$\underline{\,\,\,}$Kelly@cnu.edu}

\author{John~Osborn}
\address{Department of Mathematics, Baylor University, One Bear Place \#97328,      
 Waco, TX  76798, USA}
\email{John$\underline{\,\,\,}$Osborn@baylor.edu}

\begin{document}

\date{\today}

\begin{abstract}
We generalize the representations of $X_1$ exceptional orthogonal polynomials through determinants of matrices that have certain adjusted moments as entries. We start out directly from the Darboux transformation, allowing for a universal perspective, rather than dependent upon the particular system (Jacobi or Type of Laguerre polynomials). We include a recursion formula for the adjusted moments and provide the initial adjusted moments for each system.

Throughout we relate to the various examples of $X_1$ exceptional orthogonal polynomials. We especially focus on and provide complete proofs for the Jacobi and the Type III Laguerre case, as they are less prevalent in literature.

Lastly, we include a preliminary discussion explaining that the higher co-dimension setting becomes more involved. The number of possibilities and choices is exemplified, but only starts, with the lack of a canonical flag.
\end{abstract}

\maketitle
\section{Introduction}
The study of exceptional orthogonal polynomials arose as the result of extending exactly solvable and quasi-exactly solvable potentials in quantum mechanics beyond the Lie-algebraic setting \cite{Dutta-Roy, grandati11, grandati12,Quesne, Quesne2}.  The Laguerre and Jacobi exceptional polynomial systems of codimension one were first introduced in 2009 as an extension of Bochner's classification theorem for the classical orthogonal polynomials \cite{KMUG1}.  At that time, the approach to exceptional orthogonal polynomial systems was as state-preserving solutions to second-order Sturm-Liouville-type problems.  Quesne initially identified the relationship between the exceptional orthogonal polynomials and the Darboux transforms \cite{Quesne}.  This connection with the Darboux transform has allowed for exceptional polynomial systems of higher codimension to be obtained.  Ultimately, a new Bochner-like classification theorem for the exceptional systems has been proven \cite{KMUG9, KMUG3, KMUG5, Classification, Odake-Sasaki1}.  

Of interest to mathematicians are the various properties of the exceptional orthogonal systems as they relate to classical orthogonal systems, as well as the asymptotic and interlacing properties of the zeros, recursion formulas, and the spectrum of the exceptional systems \cite{Atia-Littlejohn-Stewart, Duran, KMUG5, KMUG6, Zeros, Ho,HoOdakeSasaki, HoSasakiZeros2012, LLMS, LLS, LLSW}. 

Here we use the Darboux transform and its relation to the exceptional Laguerre (Types I, II, and III) and Jacobi polynomial systems to obtain the corresponding exceptional polynomials via the Gram-Schmidt method.  In the case of classical orthogonal polynomials it is a well-established perspective to view the polynomials as the result of applying Gram-Schmidt to the \emph{flag sequence} $\left\{1, x, x^2, \ldots\right\}$, as is outlined in \cite[Chapter 1.3]{Chihara}. Some of the authors \cite{Liaw-Osborn} used these ideas to derive two representations for the Type I $X_1$-Laguerre polynomials in terms of their moments through using determinants. An adaptation and generalization of this method leads us to our main result, which relies solely on the moment functions and modified weights of the exceptional orthogonal polynomial systems.  Our results are universal in the sense that they can be applied to $X_1$ orthogonal polynomial systems (independent of being ``Laguerre'' or ``Jacobi'' specific). Due to structural reasons, there are no $X_1$-Hermite polynomials \cite{Hermite1}.

Other representations exceptional orthogonal polynomial involve Wronskians, and sometimes pseudo-Wronskian, determinants of classical orthogonal polynomials, see e.g.~\cite{Duran, Hermite1}.

The idea of this paper hinges on the following simple observations regarding the essential characteristics of exceptional orthogonal polynomials. First of all, the exclusion of eigenfunctions with certain degrees is caused by particular rational-function coefficients in the differential expression. When we apply the differential expression to eigenfunctions we must (at least) cancel any denominators introduced by these rational-function coefficients. This idea leads to what we call the \emph{exceptional condition}, one of the central players in the theory. In this paper we primarily consider $X_1$ polynomials, so that this denominator consists of a linear polynomial. We call the root of this polynomial the \emph{exceptional root}, and denote it by $\xi$. Our work is suspended on the ansatz that writing the Taylor expansion of the exceptional polynomials around $\xi$ inherits beneficial properties, see \cite[Section 5]{Liaw-Osborn}. There it was also noticed that adjusting moments by replacing the integrand $x^l$ by $(x-\xi)^l$ drastically simplifies matters.

We first gather some preliminary information about the Darboux transform and its relations to exceptional orthogonal polynomials (Section \ref{s-prelims}). For subsequent comparison, we introduce the Jacobi and the Type III Laguerre case in more detail. 
In Section \ref{Exceptional Condition}, we focus on a universal expression for the exceptional condition as it was obtained recently in \cite{Classification}. We write their condition in terms of the basic functions of the theory, and relate the condition to our examples. The flag sequence, which after the Gram-Schmidt algorithm returns the exceptional orthogonal polynomials, is the topic of Section \ref{flag}. We state the determinantal representation in Theorem \ref{s-determinantal} of Section \ref{theo: determinantal}. Like in \cite{Liaw-Osborn}, we relate the degree $n$ exceptional polynomial in terms of a determinant of an $(n+1)\times (n+1)$-matrix. The first row of the matrix comes from the exceptional condition, while the second through last row contains certain adjusted moments. We then, in Section \ref{s-moments}, present a recursion formula to compute these adjusted moments. We observe a curious fact: the moment representations for the Type I and II Laguerre polynomials only differ in the exceptional condition. The computation of the initial adjusted moments is deferred to the Appendix \ref{App:AppendixA}. In Section \ref{s-Xm}, we include some preliminary observations on the flags and recursive moment formulas for $X_m$ orthogonal polynomials when $m>1$, mainly so as to indicate some difficulties we expect to encounter when extending determinantal representations to the higher co-dimension setting.  For example, flags have been explored for
$m=1$ see \cite{KMUG6}, whereas the Darboux transform is used when $m>1$.

\subsection{Notation} Most of our notation leans on the standards used in the field of exceptional polynomials. There the general trend is that the classical objects are denoted by letters without ``$\,\,\,\widehat \,\,\,\,$". For example, $T^\alpha$ represents the classical Laguerre differential expression/operator; $L_n^\alpha(x)$, the classical Laguerre polynomial of degree $n$; and $W^{(\alpha, \beta)}(x)$, for the classical Jacobi weight. Their exceptional counterparts are denoted by the respective letters but with the ``$\,\,\,\widehat \,\,\,\,$". Here a subscript is used to refer to the codimension and superscripts I through III specify with which type of exceptional Laguerre we are dealing with, e.g.~$\widehat T^{I,\alpha}_m$ stands for the Type I $X_m$-Laguerre differential operator. When we talk more generally about a weight, differential expression or a general polynomial (and not specifically about either Laguerre or Jacobi) we do not include the parameters $\alpha$ and/or $\beta$. For example we use $W$, $T$ or $p_n$ for the classical objects, and $\widehat W$, $\widehat T$ or $\widehat p_n$ when we consider the exceptional counterparts.  For the reader's convenience we include the other notation in Table \ref{tab-not}.

\begin{table}
\begin{tabular}{|c|p{5 in}|}
\hline
$\eta(x)$ & Polynomial function that occurs in the natural operator\\ \hline
$s(x)$ & Linear function that occurs in the natural operator\\ \hline
$\xi$ &  The exceptional root; that is the root of $b(x)$ or, equivalently, the root of $\phi(x)$ or that of $\eta (x)$ after the $\alpha$-shift\\ \hline
$\mathcal E_n$ & 
Span of the first $n$ exceptional orthogonal polynomials\\ \hline
$\mathcal F_n$ & 
Set of polynomials of degree $\le n$ satisfying the exceptional condition\\ \hline 
$v_k$ & Flag elements $v_k(x)=(x-\xi)^k$\\ \hline
$c_{n,i}$ &   Coefficients of the Taylor expansion $\widehat y_n(x) =\sum_{i=0}^{n}c_{n,i}(x-\xi)^i$ of the exceptional polynomial around the exceptional root\\ \hline
$\widetilde \mu_m$ & Adjusted moments $\widetilde \mu_m=\int_{I} (x-\xi)^m \widehat W(x) dx$\\ \hline
\end{tabular}

\

\caption{Other notation.}
\label{tab-not}
\end{table}
\section{Preliminaries}\label{s-prelims}

We first recall how the Darboux transform may be used to generate exceptional orthogonal polynomial systems.  For further reading on the relationship of the Darboux transform and exceptional orthogonal polynomial systems, see \cite{Classification,KMUG9,KMUG3,KMUG5}, upon which the following exposition is based.  

Suppose $T[y]$ is a second-order differential operator with rational coefficients; that is
\begin{align}
T[y]&=p(x)y''(x)+q(x)y'(x)+r(x)y(x)\,.
\end{align}  

Define the following quasi-rational functions
\begin{align}
	P(x)&=\exp\left(\int \frac{q(x)}{p(x)}\,dx\right)\,,\\
	W(x)&=\frac{P(x)}{p(x)}\,,\\
	R(x)&=r(x)W(x)\,.
\end{align}

By multiplying the eigenvalue equation $T[y]=\lambda y$ by $W(x)$, the Sturm-Liouville type equation 
\[
(Py')'+Ry=\lambda W y
\] is formed.  Thus, we refer to $W(x)$ as the weight function associated with $T$.  

For $T$ and a quasi-rational function $\phi(x)$, $\phi(x)$ is called a \textit{quasi-rational eigenfunction}  if 
\begin{equation}
	T[\phi]=\lambda \phi\, \quad \lambda \in \mathbb C\,.
\end{equation}  

In order to create the operator associated with the exceptional orthogonal polynomials, we first use the following decomposition proposition to rewrite $T$ as a composition of two first-order operators.
\begin{prop}\cite[Proposition 3.5]{Classification}\label{prop: Decomp}
	For a second-order difference operator $T[y]$ having rational coefficients, let $\phi(x)$ be a quasi-rational eigenfunction of $T$ with eigenvalue $\lambda$, and let $b(x)$ be an arbitrary, non-zero rational function.  Define rational functions 
	\begin{align}
		w &=\frac{\phi'}{\phi}\,,\\
		\widehat b & = \frac{p}{b}\,,\\
		\widehat w & = -w -\frac{q}{p}+\frac{b'}{b}
	\end{align} and first-order operators $A$ and $B$ by
	\begin{align}\label{eq: AandB}
		A[y]=b(y'-wy) & \quad \mbox{ and } \quad B=\widehat b(y'-\widehat wy)\,.
	\end{align} 
	With $A$ and $B$ as above, $T$ has a rational factorization of the form $T=BA+\lambda$.  
\end{prop}

Using the factorization of $T$ and the definitions of $A$ and $B$ given in Proposition \ref{prop: Decomp}, we will define a new operator, called the \textit{partner operator}, to be $\widehat T:=AB+\lambda$.  The rational Darboux transformation maps $T$ to $\widehat T$. Operator $\widehat T$ will also be a second-order differential operator with rational coefficients; that is
\begin{align}\label{eq: standardform}
	\widehat T[y]&=p(x)y''(x)+\widehat q(x)y'(x)+\widehat r(x)y(x)\,.
\end{align} Note that the coefficient of the second-order term is the same for both the original and partner operators.  Additionally, another Sturm-Liouville type equation is induced.

\begin{prop}\cite[Proposition 3.6]{Classification}\label{prop: Coeff}
Suppose that $T$ and $\widehat T$ are second-order differential operators with rational coefficients which are related via a rational Darboux transformation.  Then $T$ and $\widehat T$ have the same second-order coefficients, while first- and zero-order coefficients $q,\widehat q, r, \widehat r\in \mathbb Q$, and the quasi-rational weights $W(x)$ and $\widehat W(x)$ are related by
\begin{align}
\label{def: q-hat}	\widehat q &=q+p'-\frac{2pb'}{b}\,,\\ 
	\widehat r & =-p(\widehat w'+\widehat w^2)-\widehat q \widehat w \\
	&=r+q'+wp'-\frac{b'}{b}(q+p')+\left(2\left(\frac{b'}{b}\right)^2-\frac{b''}{b}+2w'\right)p\,,\mbox{ and }\\
	\widehat W & = \frac{pW}{b^2}=\frac{P(x)}{b^2}\,.\label{def: WeightFunc}
\end{align}
\end{prop}

We are interested in the Darboux transform applied to classical Bochner systems; in particular, the systems of Laguerre and Jacobi. When $T$ is defined to be one of these classical systems and $\phi$ is carefully chosen, the partner operator will produce an exceptional polynomial system operator. 

\subsection{General $X_m$ Expression}
The moment representations will be approached in the general case (without reliance on the specifics of the Laguerre or Jacobi setting) and will rely solely on the information provided via the Darboux transform and Bochner systems.  Therefore, the differential expression for $X_m$ orthogonal polynomial system will be defined by $\widehat T[\widehat y]$, with the weight function given by equation \eqref{def: WeightFunc}.  Recall that $\xi$ is the exceptional root. (That is, $\xi$ is the root of the denominator function $b(x)$ of the weight $\widehat W(x)$.)  Later, when we restrict ourselves to $m=1$, $b(x)$ will be of degree two and have one repeated root.  

Since we assume that $T$ is a classical Bochner operator and $\phi$ has been carefully chosen to produce an $X_m$ orthogonal polynomial operator, the differential equation $\widehat \ell[\widehat y]=\widehat \lambda \widehat y$ will be satisfied by a sequence of polynomials $\widehat y=\left\{\widehat y_n\right\}_{n\in \mathbb N_0\backslash A}$, where $\widehat y_n$ is of degree $n$ and $A$ is a finite set of dimension $m$, and corresponding sequence of eigenvalues $\widehat \lambda=\{\widehat \lambda_n\}_{n\in \mathbb N_0\backslash A}$.  We mention that the eigenvalues for the classical orthogonal systems and exceptional orthogonal systems are \emph{not} necessarily equal as shifting may occur under the Darboux transform.  On an open interval, $I=(a,b)$, this sequence of polynomials will satisfy the orthogonality relation
\begin{align}\label{e-orthogonality}
	\langle \widehat y_n,\widehat y_k\rangle\ci{\widehat W}
	= \int_I \widehat y_n \widehat y_k\widehat W\,dx=
	K_n \delta_{n,k}
\end{align}
where $\delta_{n,k}$ is the Kronecker-$\delta$ symbol (equal to 1 when $n=k$ and 0 when $n\neq k$).
Of course, the interval $I=(0,\infty)$ for Laguerre and $I=(-1,1)$ when we work with Jacobi systems.

\subsection{$X_m$-Laguerre Expression}\label{section: X1-Laguerre}
Here we provide a brief insight into the Type I, II, and III exceptional Laguerre orthogonal polynomial systems.  A more rigorous look at the Darboux transform applied to the classical Laguerre expression may be found in \cite{KMUG9}.  Recall the classical Laguerre differential operator 
\begin{align}\label{eq: Laguerre}
T^\alpha[y]=xy''+(-x+\alpha+1)y'\,,
\end{align} and corresponding weight function \[W^\alpha(x)=x^{\alpha}e^{-x}\,.\]  The classical Laguerre polynomials are shape-invariant under the factorizations:
\begin{align*}
T^\alpha&=B^\alpha A^\alpha\\
\widehat T^{\alpha+1}&=A^\alpha B^\alpha+1\,,
\end{align*} where 
\begin{align*}
A^\alpha(y)&=y' \quad \mbox{and}\\
B^\alpha(y)&=xy'+(\alpha+1-x)y\,. 
\end{align*}

The quasi-rational eigenfunctions and eigenvalues of $T^\alpha[y]$ are:
\[
\begin{array}{ll}
\phi_1(x)=L_m^\alpha(x),&\lambda=-m\,,\\
\phi_2(x)=x^{-\alpha}L_m^{-\alpha}(x),&\lambda=\alpha-m\,,\\
\phi_3(x)=e^xL_m^\alpha(-x),&\lambda=\alpha+1+m\,, \mbox{ and }\\
\phi_4(x)=x^{-\alpha}e^xL_m^\alpha(-x),&\lambda=m+1\,,
\end{array}
\]
where $m \in \mathbb N_0$.  The factorizations corresponding to each of these eigenfunctions have been studied, see \cite{KMUG9, Zeros}.  It has been shown that $\phi_1$ with $m=0$ corresponds to a state-deleting transformation and corresponds to the classical Laguerre polynomials.  For $m>0$, the eigenfunctions corresponding to $\phi_1$ yield singular operators, which means that no new families of orthogonal polynomials arise.  The family associated with $\phi_4$ is state-adding and therefore, the resulting orthogonal polynomials are not of codimension $m$.  The factorizations of $\phi_2$ and $\phi_3$ result in new orthogonal polynomials---in fact, these factorizations respectively produce the Type I and Type II exceptional orthogonal polynomials of codimension $m$.  For further reading regarding the properties of the Type I and Type II exceptional Laguerre polynomial systems, see  \cite{LLMS}.

The Type III exceptional Laguerre polynomials $\{\widehat L_{m,n}^{III,\alpha}(x)\}_{n=0 \text{ or }n >m}$ are a class of Laguerre-type orthogonal polynomials which were extensively studied in \cite{LLMS}.  We will focus on the Type III exceptional operator for several of our examples in this paper. 

There is another rational factorization of the classical Laguerre expression $T^\alpha[\cdot]$, which yields the Type III second-order expression.
Let 
\begin{align*}
A_{m}^{III, \alpha}(y)&=xL_m^{-\alpha}(-x)y'-(m+1)L_{m+1}^{-\alpha-1}(-x)y\,,\mbox{ and }\\
B_{m}^{III, \alpha}(y)&=\frac{y'}{L_m^{-\alpha}(-x)}\,.
\end{align*}  The classical Laguerre operator may be written as
\[
T^\alpha=B_{m}^{III,\alpha}A_{m}^{III,\alpha}+m+1\,,
\]
and the Darboux transformation associated with the above factorizations yields the Type III operator 
\begin{align*}
\nonumber \widehat T_{m}^{III, \alpha}&=A_{m}^{III, \alpha+1 }B_{m}^{III, \alpha+1}+m-\alpha.
\end{align*} 
That is, explicitly, we have
\begin{align}
 \label{eq: Type III DE} 
  \widehat T_{m}^{III, \alpha}[y]&=
  -xy''+\left(-1+\alpha+x+2x\frac{(L_m^{-\alpha-1}(-x))'}{L_m^{-\alpha-1}(-x)}\right)y'-\alpha y\,.
\end{align} 

The corresponding weight for the Type III case is 
\begin{equation}\label{Type III Weight}
\widehat W^{III,\alpha}_{m}(x)=\frac{x^\alpha e^{-x}}{\left(L_m^{-\alpha-1}(-x)\right)^2}\,.
\end{equation}  The Type III eigenvalue equation will have orthogonal polynomial solutions on $(0,\infty)$ if and only if $-1<\alpha<0$.  In fact, the associated Type III differential equation will have a polynomial solution $y(x)=\widehat L^{III, \alpha}_{m,n}(x)$ of degree $n$ for $n=0$ and for $n\geq m+1$; that is, solutions of degrees $\{1,2,\ldots, m\}$ are missing. 

We may write the Type III exceptional Laguerre polynomials using the Darboux transformation
\begin{align}\label{eq: TypeIII Lag Poly}
\widehat L^{III,\alpha}_{m,n}(x) &=
\begin{cases}
-A^{III, \alpha+1}_{m}[L^{\alpha+1}_{n-m-1}(x)], &n=m+1,m+2,\ldots\\
1 & n = 0\,.
\end{cases}
\end{align}

\subsection{$X_m$-Jacobi Expression}\label{section: X1-Jacobi}
Again, we provide a brief outline of the exceptional Jacobi orthogonal polynomial systems, and a more rigorous look at the Darboux transform applied to the classical Jacobi expression may be found in \cite{Zeros}. For $\alpha, \beta>-1$, 
\begin{align}\label{eq: ClassicalJacobi}
T^{\alpha, \beta}[y]&=(1-x^2)y''+(\beta-\alpha-(\alpha+\beta+2)x)y'
\end{align} is the classical Jacobi differential operator.  For \[\alpha, \beta >-1,\quad \alpha+1-m-\beta\notin\left\{0,1,\ldots, m-1\right\},\quad \text{and}\quad \text{sgn}(\alpha+1-m)=\text{sgn}(\beta),\] define 
\begin{align*}
A^{\alpha, \beta}_m[y]&= (1-x)P_m^{(-\alpha,\beta)}y'+(m-\alpha)P_m^{(-\alpha,\beta)}y\,,\\
B^{\alpha, \beta}_m[y]&=\frac{(1+x)y'+(1+\beta)y}{P_m^{(-\alpha,\beta)}}\,,
\end{align*} where $P_m^{(\alpha, \beta)}$ is the classical Jacobi polynomial of degree $m$.  It follows that 
\begin{align*}
T^{\alpha, \beta}[y]&=B^{\alpha, \beta}_mA^{\alpha, \beta}_m y-(m-\alpha)(m+\beta +1)y
\end{align*} and the exceptional Jacobi operator is defined as
\begin{align}
\notag
\widehat T^{\alpha, \beta}_m[y]&= A^{\alpha+1,\beta-1}_mB^{\alpha+1, \beta-1}_m y-(m-\alpha -1)(m+\beta)y\\
 & =T^{\alpha, \beta}[y]+(\alpha-\beta-m+1)my-(\log P_m^{(-\alpha-1,\beta-1)})'\left(\beta(1-x)y+(1-x^2)y'\right)
 \label{e-XmJacobi}
\end{align}

For $\alpha, \beta >-1$ and $n\geq m$, the $X_m$-Jacobi polynomial of degree $n$ can be written as
\begin{align*}
\widehat P_{m,n}^{(\alpha, \beta)}(x)& = \frac{(-1)^{m+1}}{\alpha+1+j}A^{\alpha+1,\beta-1}_m\left[P_j^{(\alpha+1,\beta-1)}(x)\right]\,, \quad j=n-m\geq 0\,.
\end{align*} Note that the exceptional operator extends the classical Jacobi operator; that is \[\widehat T^{\alpha, \beta}_0[y]=T^{\alpha, \beta}[y]\,.\]  The $X_m$-Jacobi polynomials $\{\widehat P_{m,n}^{(\alpha, \beta)}\}_{n\geq m}$ are orthogonal on (-1,1) with respect to the weight function  \[\widehat W_m^{\alpha, \beta}(x)=\frac{(1-x)^\alpha (1+x)^\beta}{\left(P_m^{(-\alpha-1,\beta-1)}(x)\right)^2}\,.\]

\section{Exceptional Condition}\label{Exceptional Condition}
The most noticeable difference between a classical orthogonal polynomial expression, as classified by Bochner, and the exceptional orthogonal polynomial expressions is that the coefficient functions for the first- and zero-order terms are no longer polynomial.  From Proposition \ref{prop: Coeff}, we see that the coefficient functions for any second-order partner operator formed via a rational Darboux transformation will have denominators containing powers of $b(x)$ and $\phi(x)$.  

We turn our attention to our particular case, where we are working with $\widehat T$, an exceptional polynomial operator.  By \cite[Definition 7.1, ii-c]{Classification}, $f(x)p(x)\widehat W(x)\rightarrow 0$ at the endpoints of $(a,b)$ for every polynomial $f(x)$.  Consequently, the operator $\widehat T$ will be polynomially regular and thus, semi-simple \cite[Remark 4.10]{Classification}.  In other words, since we seek polynomial solutions to the associated eigenvalue problem, these non-polynomial coefficients require a specific structure condition for the remaining polynomials; that is, ``cancellation'' must occur in order to have polynomial eigenfunctions. The coefficient functions which are non-polynomial form the \textit{exceptional term}, and the specific structure induced on solutions that is referred to as the \textit{exceptional condition}.  Polynomials of every degree cannot satisfy both the exceptional condition and form a maximal invariant subspace under $\widehat T$.  Therefore, we do not have a full sequence of polynomial eigenfunctions for the exceptional operators---that is $A$ is ensured to be non-empty.

In order to find the exceptional condition which characterizes the exceptional polynomial systems we introduce an additional setting in which we can consider the exceptional orthogonal polynomial operators.  We say that for any two second-order operators having rational coefficients, $T$ and $\widehat T$ are \textit{gauge-equivalent} if there exists a rational function $\sigma$ such that \[\sigma T=\widehat T \sigma\,.\]  By \cite[Proposition 2.5, Theorem 5.4]{Classification}, every exceptional operator will be gauge equivalent to a natural operator.  A natural operator is a second-order operator of the form
\begin{equation}\label{def: naturalform}
py''+\left(\frac{p'}{2}+s-\frac{2p\eta'}{\eta}\right)y'+\left(\frac{p\eta''}{\eta}+\left(\frac{p'}{2}-s\right)\frac{\eta'}{\eta}\right)y\,,
\end{equation} where $p$ is a second-degree polynomial, $\eta$ is a polynomial, and $s$ is a linear function.  It is the case that the polynomial $p$ is the same regardless of whether $\widehat T$ is written in standard form \eqref{eq: standardform} or in natural gauge form \eqref{def: naturalform}.  Obviously, if the standard form of $\widehat T$ is known (that is, $b$ is known), by setting the coefficients to be equal, one can find the polynomial $s$. Rather, we would like to approach the task of finding $s$ without knowing $b$.  To do this, we will utilize \cite[Corollary 5.25]{Classification} which states (in paraphrased form) that the exceptional term for $\widehat T$ in the natural gauge is given by
\begin{equation}\label{def: exceptionalterm}
\frac{2p\eta'y'-\left(p\eta''+\displaystyle{\frac{p'\eta'}{2}}-s\eta'\right)y}{\eta}\,.
\end{equation}  Finding the linear polynomial $s$ given the exceptional terms from the two representations of $\widehat T$ will be the focus Subsection \ref{Linear Poly}.

\subsection{Finding the Linear Polynomial $s$}\label{Linear Poly}
We now turn our attention to finding the linear polynomial $s$, which is found in \eqref{def: naturalform}.

While it is clear that $\eta$ should be the polynomial part of the quasi-rational eigenfunction $\phi$ (that initiates the Darboux transform), it seems less obvious what the function $s\in \mathcal{P}_1$ would be. Here $\mathcal{P}_1$ denotes the polynomials of degree less than or equal to one. Of course, one could extract $s$ by comparing the coefficient of $y'$ in $\widehat T$ to $\widehat q$ in \eqref{def: q-hat} after having computed everything. But we found it beneficial to express $s$ directly from $\eta$ and the factorization gauge.

\begin{lem}\label{l-s}
	The linear polynomial $s = q + p'/2 + 2p (\eta'/\eta - b'/b).$
\end{lem}

From this formula, it is not immediately clear that $s\in \mathcal{P}_1$ in general.  However, this is exactly the topic of \cite[Section 5]{Classification}. Below we include the values for $s$ for the Type I, II, and III Laguerre and for the Jacobi case.

\begin{proof}
	We set the coefficient of $y'$ in \eqref{def: naturalform} equal to $\widehat q$ in Proposition \ref{prop: Coeff}
	\[
	-\frac{2\eta'p}{\eta} + \frac{p'}{2} +s
	=q+p'-\frac{2b'p}{b}.
	\]
	Solving for $s$ yields the lemma.
\end{proof}

In Table \ref{tab-one} we include a table for the choices of $\eta$ and $s$ for the three types of Laguerre examples. We took the liberty of already including all the shifts.  As required, we observe that $s\in \mathcal{P}_1$ in all cases.

We exemplify how we these expressions play out in the case of the Type III Laguerre case.

\begin{ex*}
	First recall that the coefficients of the classical Laguerre expression \eqref{eq: Laguerre} are $p(x) = -x$, $q(x) = x-\alpha-1$ and $r(x)\equiv 0$. The Type III Laguerre orthogonal polynomial system is derived from the classical Laguerre system by using the quasi-rational eigenfunction $\phi_4(x) = x^{-\alpha} e^{x} L^{-\alpha}_m(-x)$, so that
	$
	\widetilde\eta(x): = L^{-\alpha}_m(-x).
	$
	In \cite{LLMS} we used the factorization gauge $b(x) = xL^{-\alpha}_m(-x)$. We also use the shift $\alpha\mapsto \alpha+1$. This shift can be seen from the $\alpha+1$ superscript in the partner operator $$-\widehat T ^{\textit{III},\alpha}_m = A^{\textit{III},\alpha+1}_m\circ B^{\textit{III},\alpha+1}_m +m-\alpha.$$
	
	We compute
	\[
	\frac{\widetilde\eta'(x)}{\widetilde\eta(x)} = \frac{(L^{-\alpha}_m(-x))'}{L^{-\alpha}_m(-x)}
	\qquad
	\text{and}
	\qquad
	\frac{b'(x)}{b(x)} = \frac{1}{x}+\frac{(L^{-\alpha}_m(-x))'}{L^{-\alpha}_m(-x)}\,,
	\]
	so that the expression in Lemma \ref{l-s}
	\[
	q + \frac{p'}{2} + 2p \left(\frac{\widetilde\eta'}{\widetilde\eta} - \frac{b'}{b}\right)
	=
	x-\alpha-1-\frac{1}{2}-2x\left(-\frac{1}{x}\right) = x-\alpha+\frac{1}{2}\,.
	\]
	
	And after the shift $\alpha\mapsto\alpha+1$ we obtain
	\[
	\eta(x) = L^{-\alpha-1}_m(-x)
	\qquad
	\text{and}
	\qquad
	s(x) = x-\alpha-\frac{1}2.
	\]
	
	We verify that the coefficient of $y'$ is correct by comparing
	\begin{align*}
	\left(\frac{p'}{2}+s-\frac{2p\eta'}{\eta}\right)y'
	&=
	\left(-\frac{1}2 +x-\alpha -\frac{1}2 - 2x \frac{(L^{-\alpha-1}_m(-x))'}{L^{-\alpha-1}_m(-x)}\right)y'\\
	&=
	\left(x-\alpha -1 - 2x \frac{(L^{-\alpha-1}_m(-x))'}{L^{-\alpha-1}_m(-x)}\right)y'
	\end{align*}
	with the $X_m$ expression \eqref{eq: Type III DE} for Type III Laguerre.
	
We also compute the coefficient of the zero-order term to ensure that this choice for $s$ produces an equivalent expression in standard form.  Using $\eta=L_m^{-\alpha-1}(-x)$ and the coefficient for $y$ from \eqref{def: naturalform}, we have the following calculation, which relies on the fact that  $\eta=L_m^{-\alpha-1}(-x)$ is a solution to the classical Laguerre differential equation $ T [\eta]=m \eta$, where $ T $ is defined in \eqref{eq: Laguerre}:
\begin{align*}
\frac{\left(p\eta''+\displaystyle{\frac{p' \eta'}{2}}-s\eta'\right)}{\eta} & =\frac{-x\left(L_m^{-\alpha-1}(-x)\right)''+\left(-\displaystyle{\frac{1}{2}}-x+\alpha+\displaystyle{\frac{1}{2}}\right)\left(L_m^{-\alpha-1}(-x)\right)'}{L_m^{-\alpha-1}(-x)}\\
&=\frac{-x\left(L_m^{-\alpha-1}(-x)\right)''+(-x+\alpha)\left(L_m^{-\alpha-1}(-x)\right)'}{L_m^{-\alpha-1}(-x)}\\
&=\frac{ m L_m^{-\alpha-1}(-x)}{L_m^{-\alpha-1}(-x)}\\
&=m.
\end{align*}  Note that the coefficient of the zero-order term we found using the natural operator and $s$ differs from the coefficient in the standard form given in \eqref{eq: Type III DE} by a constant.  This is a result of the shifting of eigenvalues which occurs when the Darboux transformation is applied.  We are not concerned by this discrepancy as the structure of the operator remains the same.
\end{ex*}

\begin{table}\scriptsize
\begin{tabular}{| l | l | l | l | l | l |}
\hline& $\phi(x)$&$\eta(x)$&$b(x)$& Shift&$s(x)$\\\hline
\text{Type I Lag.}&$e^x L^{\alpha}_m(-x)$&$L^{\alpha}_m(-x)$&$L^{\alpha}_m(-x)$&$\alpha\mapsto\alpha-1$&$x-\alpha-1/2$\\\hline
\text{Type II Lag.}&$x^{-\alpha}L^{-\alpha}_m(x)$&$L^{-\alpha}_m(x)$&$xL^{-\alpha}_m(x)$&$\alpha\mapsto\alpha+1$&$x-\alpha-1/2$\\\hline
\text{Type III Lag.}&$x^{-\alpha} e^xL^{-\alpha}_m(-x)$&$L^{-\alpha}_m(-x)$&$xL^{-\alpha}_m(-x)$&$\alpha\mapsto\alpha+1$&$ x-\alpha-1/2$\\\hline
Jacobi& $(1-x)^{-\alpha}\cdot$& $P_m^{(-\alpha, \beta)}(x)$ & $(1-x)\cdot$ & $\alpha\mapsto\alpha+1$ & $\beta-\alpha$\\
&$\quad {P_m^{(-\alpha, \beta)}}(x)$&&$\quad P_m^{(-\alpha, \beta)}(x)$&$\beta\mapsto\beta-1$& $\quad -(\alpha+\beta+1)x$\\ \hline
\end{tabular}

\

\caption{Choices of $\eta$ and $s$ for the natural operator in equation \eqref{def: naturalform}.}\label{tab-one}
\end{table}

\begin{ex*}
The Jacobi expression, $\phi(x)=(1-x)^{-\alpha}P^{(-\alpha,\beta)}_m(x)$, is the quasi-rational eigenfunction is. Using Lemma \ref{l-s} together with the function $\eta(x) = P^{(-\alpha,\beta)}_m(x)$, the gauge $b(x) = (1-x)P^{(-\alpha,\beta)}_m(x)$ and the shifts $\alpha\mapsto\alpha+1$ and $\beta \mapsto \beta-1$ yield
\[
s(x)
=
\beta-\alpha - (\alpha+\beta+1)x.
\]

We verify that the coefficient of $y'$ is correct by computing
	\begin{align*}
	\left(\frac{p'}{2}+s-\frac{2p\eta'}{\eta}\right)y'
	&=
	\left(-x+\beta-\alpha -(\alpha+\beta+1)x - 2(1-x^2) \frac{(P^{(-\alpha-1,\beta-1)}_m(x))'}{P^{(-\alpha-1,\beta-1)}_m(x)}\right)y'\\
	&=
	\left(\beta-\alpha -(\alpha+\beta+2)x - 2x \frac{(P^{(-\alpha-1,\beta-1)}_m(x))'}{P^{(-\alpha-1,\beta-1)}_m(x)}\right)y'\\
	\end{align*}
	with the $X_m$ expression for Jacobi given by \eqref{eq: ClassicalJacobi} and \eqref{e-XmJacobi}.
\end{ex*}
\section{The Flag}\label{flag}
In this section, we characterize the subspace spanned by the first $n$ of the $X_1$-exceptional orthogonal polynomials as those polynomials are satisfying the exceptional condition 
	\begin{equation}\label{eq: exceptional condition}
	\left[2p\eta'y'-\left(p\eta''+\frac{p'\eta'}{2}-s\eta'\right)y\right]\bigg|_{x=\xi}=0\,.
	\end{equation}	
Recall that for the exceptional orthogonal systems of codimension one, $\xi$ represents the root of $\eta$ (that is the exceptional root).  The first $m$ $X_1$-exceptional orthogonal polynomials will be those of degree less than or equal to $m$ in the case of the Type I and Type II $X_1$-Laguerre or $X_1$-Jacobi polynomial systems.  For the Type III $X_1$-Laguerre polynomials system, the first $m$ polynomials will be those of degree 0 or between 2 and $m$, inclusively.

%
In preparation for Lemma \ref{lemma} below, we present two definitions.  We present the definitions as applicable for the Type I and II Laguerre and Jacobi systems.  Modifications required for the Type III Laguerre are noted in parentheses.  Let $\mathcal P_n$ denote the set of polynomials of degree less than or equal to $n$ 
and define the span of the first $n$ exceptional orthogonal polynomials 
\[
\mathcal E_n:=\text{span}\left\{\widehat y_j: j=1,\ldots, n\right\},
\] where $\widehat y_j$ is defined as in Table \ref{tab-two}.
\begin{table}
\begin{tabular}{|c|c|l|}
			\hline
			\multicolumn{1}{ |c  }{\multirow{2}{*}{Exceptional Laguerre} } &
			\multicolumn{1}{ |c| }{Type I } & $\widehat y_j=\widehat L_{1,j}^{I,\alpha}$ \\ 
			\cline{2-3}
			& \multicolumn{1}{ |c| }{Type II } & $\widehat y_j=\widehat L_{1,j}^{II,\alpha}$ \\ 
			\cline{2-3}
			& \multicolumn{1}{ |c| }{Type III } & $\widehat y_j=\left\{\begin{array}{ll} \widehat L_{1,j}^{III,\alpha} & j\geq 2\\
			1 & j=0\end{array} \right.$\\
			\hline
			Exceptional Jacobi& &$\widehat y_j=\widehat P_{1,j}^{(\alpha, \beta)}$ \\
			\hline
		\end{tabular}

\

\

\caption{In the definition of $\mathcal{E}_n=\text{span}\left\{\widehat y_j: j=1,\ldots, n\right\}$ as the first $n$ exceptional polynomials we have $\widehat y_j$.}\label{tab-two}
	\end{table}
	
Further we define $$\mathcal F_n:=\left\{p \in \mathcal P_n: p\text{ satisfies } \eqref{eq: exceptional condition}\right\}.$$
\begin{lem}
	\label{lemma}
	The sets $ \mathcal E_n = \mathcal F_n $ for all $n\in \N$.
\end{lem}

The proof is analogous to the proof for \cite[Lemma 2.1]{Liaw-Osborn}, but the exceptional condition is now replaced by a universal one.  We include the argument for the convenience of the reader.

\begin{proof} 
	Since $\mathcal E_n$ and $\mathcal F_n$ are clearly vector spaces of equal dimension $n$, it suffices to show that $ \mathcal E_n \subseteq \mathcal F_n $.
	
	To see this, take $f \in \mathcal E_n$.  Then $f$ is a linear combination of the first $n$ exceptional polynomials.  So $f \in \mathcal P_n$.
	Further, $\widehat T [f]$ is polynomial, and so is the expression \eqref{eq: exceptional condition} for $y = f$.  It follows that $f \in \mathcal F_n$.
\end{proof}
With the exceptional root $\xi$, we define degree $k$ polynomials
\begin{align}\label{e-flag}
v_k(x) := (x-\xi)^k
\quad
\text{for }k\ge 2.
\end{align}

\begin{lem}\label{l-flag}
	The sequence of polynomials $\{v_1, v_2, v_3, \hdots\}$ forms a flag for $\widehat T$. 
\end{lem}

\begin{proof} 
	Our definition of $v_k$ ensures that the polynomials have the appropriate degrees.  That is, for Types I and II $X_1$-Laguerre and $X_1$-Jacobi systems, the constant polynomial is missing, and for Type III Laguerre the linear polynomial is excluded.
	
	In virtue of the previous lemma, it suffices to prove that all $v_k$ satisfy the exceptional condition \eqref{eq: exceptional condition}.  This follows immediately for the first flag element, as $\widehat y_1 = v_1$ is the first exceptional polynomial and hence an eigenfunction of $\widehat T$.
	
	For $k \ge 2$, we recall that $v_k(x) := (x-\xi)^k$. In particular, we have $v_k(\xi) = 0$ so that the second term of \eqref{eq: exceptional condition} vanishes. On the left-hand side of \eqref{eq: exceptional condition} we are left with $2p(\xi)\eta^\prime(\xi)v_k^\prime(\xi)$. But we also have $v_k^\prime(\xi) = \left. k(x-\xi)^{k-1} \right|_{x=\xi} = 0$, because $k \ge 2$. So the exceptional condition \eqref{eq: exceptional condition} is satisfied and $\mathcal F_n = \{v_1, v_2, v_3, \hdots, v_n\}$.
\end{proof}
\section{Determinantal Representations}\label{s-determinantal}
In this section, we provide the details for finding the determinantal representations for the Type II and III $X_1$-Laguerre and $X_1$-Jacobi orthogonal polynomial systems.  In the case of the Type I $X_1$-Laguerre polynomials, we do confirm that our results agree with \cite{Liaw-Osborn} and refer the reader to \cite{Liaw-Osborn} for the details of that particular case. It is the case that $\eta$ will have one and only one exceptional root, $\xi$.  The exceptional condition as discussed in Sections \ref{Exceptional Condition} and \ref{flag} and in \cite[Corollary 5.25]{Classification} reduces to \eqref{eq: exceptional condition}.
	
To find the first row of entries in the determinantal representation, we follow the methods outlined in \cite{Liaw-Osborn}, we use the ansatz for degree $n$, $n\ge 2$, to write the exceptional polynomial
	\begin{align}\label{e-ansatz}
	\widehat y_n(x) :=
	\sum_{i=0}^{n}
	c_{n,i}(x-\xi)^i.
	\end{align}
Note that $\widehat y_n(x)$ may be any $X_1$-Laguerre-type or $X_1$-Jacobi polynomial of degree $n$.  We fill in specific details for each case below.  
	
Then
	\begin{align}\label{e-deransatz}
	\widehat y_n\,'(x) :=
	\sum_{i=1}^{n}
	ic_{n,i}(x-\xi)^{i-1}.
	\end{align} and $\widehat y_n(\xi)=c_{n,0}$ and $\widehat y_n\,'(\xi)=c_{n,1}$. 

In order to obtain the determinantal representation we notice that the coefficients $c_{n,i}$, $i=0,1,\ldots, n$, are given as the unique solution of a system of $n+1$ linear equations with matrix form $A c = b$. The objects $A$, $c$ and $b$ are given below.

To this end we define the \emph{adjusted moments}
  \begin{equation}\label{eq: d-moments}
  \widetilde \mu_m:=
  \int\ci{I} (x-\xi)^m \widehat W(x) dx
  \end{equation}

and the vectors
\[
c := (c_{n,0},c_{n,1}, \hdots ,c_{n,n})^\top
\in \R^{n+1}
\quad \text{and}\quad
b := 
(0, \hdots ,0,K_n)^\top
\in \R^{n+1}.
\]
The constant $K_n:=\langle \widehat p_n, \widehat p_n\rangle$ determines the normalization of the exceptional polynomials.

\begin{rem*}
	For the reader's reference we also include the square, $K_n$, of the norms for each of the exceptional sequences in Table \ref{tab-norms}.  The norms for the $X_m$-Laguerre polynomials are found in \cite{LLMS}; for Jacobi, see \cite{KMUG1}.
\end{rem*}

	\begin{table}
		\begin{tabular}{|c|c|l|}
			\hline
			\multicolumn{1}{ |c  }{\multirow{2}{*}{Exceptional Laguerre} } &
			\multicolumn{1}{ |c| }{Type I } & $K_n=\frac{(\alpha+n)\Gamma(\alpha+n-1)}{(n-1)!}$ for $n\geq 1$ \\ 
			\cline{2-3}
			& \multicolumn{1}{ |c| }{Type II } & $K_n=\frac{(\alpha+n-1)\Gamma(\alpha+n+1)}{(n-1)!}$ for $n\geq 1$ \\ 
			\cline{2-3}
			& \multicolumn{1}{ |c| }{Type III } & $K_n=\left\{\begin{array}{cc} \frac{n\Gamma(n+\alpha)}{(n-2)!} & n \geq 2\\
			\frac{\Gamma(\alpha+1)\Gamma(-\alpha)}{\Gamma(1-\alpha)} & n=0\end{array} \right.$\\
			\hline
			Exceptional Jacobi& &$K_n=\frac{2^{\alpha+\beta+1}(\alpha+n)(\beta+n)\Gamma(\alpha+n)\Gamma(\beta+n)}{4(\alpha+n-1)(\beta+n-1)(\alpha+\beta+2n-1)\Gamma(n)\Gamma(\alpha+\beta+n)}$ \\
			\hline
		\end{tabular}
		
		\
		
		\caption{Square of the norms of the $X_1$ polynomials.}
		\label{tab-norms}
	\end{table}

\begin{theo}\label{theo: determinantal}
The $X_1$ orthogonal polynomials have the determinantal representation formula
\begin{align}\label{e-repr}
\widehat y_n(x)
&=\frac{1}{\det A}\sum_{i=0}^n\left(\det A_i\right)(x-\xi)^i\\ \label{e-repre}
&=\frac{K_n}{\det A} \left|\begin{array}{c}\mbox{First } n\mbox{ rows of the matrix } A\\ 1 \quad \quad  (x-\xi) \quad \quad  (x-\xi)^2  \quad \quad \hdots  \quad \quad  (x-\xi)^n\end{array}\right|
\end{align}
where the $(n+1)\times(n+1)$-matrix $A$ is given by \footnotesize{
\[
A = \left[
\begin{array}{ccccc}
p(\xi)\eta''(\xi)+\frac{p'(\xi)\eta'(\xi)}{2}-s(\xi)\eta'(\xi)& 2p(\xi)\eta'(\xi) & 0 & \hdots & 0\\
c_{1,0}\widetilde \mu_0+c_{1,1}\widetilde\mu_1 & c_{1,0}\widetilde \mu_1+c_{1,1}\widetilde \mu_2&  c_{1,0}\widetilde \mu_2+c_{1,1}\widetilde \mu_3 & \hdots & c_{1,0}\widetilde \mu+c_{1,1}\widetilde \mu_{n+1}\\
\widetilde \mu_2 & \widetilde \mu_{3} &\widetilde \mu_{4}&\hdots & \widetilde \mu_{n+2}\\
\vdots & \vdots & \vdots &  & \vdots\\
\widetilde \mu_n & \widetilde \mu_{n+1} &\widetilde \mu_{n+2}& \hdots& \widetilde \mu_{2n}\\
\end{array}
\right],\] }
\normalsize
the adjusted moments are defined to be $$\widetilde \mu_m=\int_I(x-\xi)^m\widehat W(x)\,dx,$$ and where the matrix $A_k$ is obtained from $A$ by replacing the $(k+1)$-st column with the vector $b$. 
\end{theo}

In Section \ref{s-moments}, we work out recursion formulas for the adjusted moments, $\widetilde\mu_k$, of each $X_1$ family.  In addition, we compute $L_{1,2}^{III,\alpha}(x)$ as an example.

In Table \ref{tab-cs} we present the specifics for the matrix corresponding to each of the exceptional cases, noting that the constants $c_{1,0}$ and $c_{1,1}$ are easily obtained by comparing the formula \eqref{e-ansatz} with the table of $v_1=\widehat y_1$ from Section \ref{flag}. 

	\begin{table}
		\begin{tabular}{cc|c|c|}
			\cline{3-4}
			&&$c_{1,0}$ & $c_{1,1}$\\
			\hline
			\multicolumn{1}{ |c  }{\multirow{2}{*}{Exceptional Laguerre} } &
			\multicolumn{1}{ |c| }{Type I } & 1& 1\\ 
			\cline{2-4}
			\multicolumn{1}{ |c| }{}& \multicolumn{1}{ |c| }{Type II } & 1 & 1 \\ 
			\cline{2-4}
			\multicolumn{1}{ |c| }{}	& \multicolumn{1}{ |c| }{Type III } & 1 & 0 \\
			\hline
			\multicolumn{1}{ |c| }{Exceptional Jacobi}& &$\frac{\alpha+\beta}{\beta-\alpha}$ & $\frac{1}{2}$\\
			\hline
		\end{tabular}
		
		\
		
		\
		
		\caption{The values of the constants $c_{1,0}$ and $c_{1,1}$ in the matrix $A$. At the same time $c_{1,0}$ and $c_{1,1}$ are also the coefficients of $\widehat y_1(x)$, which is defined in \eqref{e-ansatz}.}
		\label{tab-cs}
	\end{table}

\begin{proof}[Proof of Theorem \ref{theo: determinantal}]
We fill the matrix $A$ row-wise.

Substituting the ansatz \eqref{e-ansatz} and its derivative \eqref{e-deransatz} into the exceptional equation \eqref{eq: exceptional condition}, leads to 
\[
\left[2p\eta'c_{n,1}-\left(p\eta''+\frac{p'\eta'}{2}-s\eta'\right)c_{n,0}\right]\bigg|_{x=\xi}=0\,.\]
We collect this information in the first row of the determinantal representation. Since our matrix equation has the form $A c = b$, the first row of the matrix $A$ now reads
\[
\left[p(\xi)\eta''(\xi)+\frac{p'(\xi)\eta'(\xi)}{2}-s(\xi)\eta'(\xi)\quad\quad 2p(\xi)\eta'(\xi) \quad\quad 0 \quad\quad \hdots \quad\quad 0\right]\,.
\]

The other $n$ rows come from the orthogonality relations \eqref{e-orthogonality}. These conditions inform us that not only $\widehat y_n\perp\widehat y_k$ for $n\neq k$, but also that $\widehat y_n\perp \mathcal E_m$ for $m<n$. Since $v_m\in  \mathcal E_m$, the orthogonality conditions imply
\begin{align*}
\langle \widehat y_n, v_k\rangle\ci{\widehat W}
&=
0
\quad
\text{for }1\le k< n, \text{ and}\\
\langle \widehat y_n, \widehat y_n\rangle\ci{\widehat W}
&=
K_n
.
\end{align*}

The second row of the matrix $A$ is obtained by substitution of \[v_1(x) = \widehat p_1=c_{1,0}+c_{1,1}(x-\xi)\] and the ansatz \eqref{e-ansatz} into this orthogonality relation. Using $v_1$ and the adjusted moments \eqref{eq: d-moments}, we compute
\begin{align*}
0 = \langle \widehat p_n, v_1\rangle\ci{\widehat W} &= \sum_{i=0}^n c_{n,i}\langle (x-\xi)^i, v_1\rangle\\
&=\sum_{i=0}^n c_{n,i}\left[\langle (x-\xi)^i,c_{1,0}\rangle +c_{1,1}\langle (x-\xi)^i,(x-\xi)\rangle\right]\\
&=\sum_{i=0}^n c_{n,i}\left[c_{1,0}\widetilde \mu_i+c_{1,1}\widetilde\mu_{i+1}\right]\,.
\end{align*}
Thus, the second row of the determinantal representation is
\[
\left[c_{1,0}\widetilde \mu_0+c_{1,1}\widetilde\mu_1 \quad \quad c_{1,0}\widetilde \mu_1+c_{1,1}\widetilde \mu_2 \quad \quad \hdots \quad \quad c_{1,0}\widetilde \mu_n+c_{1,1}\widetilde \mu_{n+1}\right]
\]

For $2\leq k\leq n$ the $(k+1)$-st row of matrix $A$ is found in analogy. Namely, recalling the definition of $v_k = (x-\xi)^k$ for $k\ge 2$, we compute
\begin{align*}
0 = \langle \widehat p_n, v_k\rangle\ci{\widehat W} &= \sum_{i=0}^n c_{n,i}\langle (x-\xi)^i, v_k\rangle_{\widehat W}\\
&=\sum_{i=0}^nc_{n,i}\widetilde \mu_{i+k}\,.
\end{align*}
Thus, row $l=k+1$, $3\leq l\leq n+1$, is given by
\[
\left[\widetilde \mu_k \quad \quad \widetilde \mu_{k+1} \quad \quad \hdots \quad \quad \widetilde \mu_{n+k}\right]\quad \quad \mbox{or}
\] 
\[
\left[\widetilde \mu_{l-1} \quad \quad \widetilde \mu_{l} \quad \quad \hdots \quad \quad \widetilde \mu_{n+l-1}\right]\,.
\]

Equation \eqref{e-repr} now follows from Cramer's rule, and \eqref{e-repre} comes about from the co-factor definition of determinants, upon expanding \eqref{e-repre} along the last row.
\end{proof}

{
\section{Recursion Relations for the Adjusted Moments}\label{s-moments}
The exceptional polynomials $\widehat y_n$ may be represented using the adjusted moments $\widetilde \mu_k$.  Thus, we can develop a recursive formula for the moments via the three-term recursion relation associated with the polynomial sequence.  In Table \ref{tab-recursion} we present the recursive formulas for all $X_1$ orthogonal polynomial systems. We include the proofs for both the Type III Laguerre and Jacobi moments.  The proofs follow in analogy to the recursive formula of the Type I moments found in \cite{Liaw-Osborn}.

Prior to stating Theorem \ref{theo: Recursion formula} we note that to simplify notation, we allow $\widehat W$ and $I$ to respectively represent the appropriate weight function and interval of orthogonality pertaining to the each exceptional system.

\begin{theo}\label{theo: Recursion formula}
For an $X_1$  orthogonal polynomial sequence satisfying 
\[a_2y\,''+a_1y\,'+a_0 y=\lambda y\,\] the adjusted moments $\widetilde \mu_k=\int_{I} (x-\xi)^k \widehat W(x)\,dx$ satisfy the recursion formulas for $k \in \mathbb{N}_0$:
\[\widetilde{\mu}_{k+1}=-(r_2+s_1)^{-1}\left[(kr_0+s_{-1})\widetilde{\mu}_{k-1}+(kr_1+s_0)\widetilde{\mu}_{k}\right]\,,\]
where $a_2(x)=\sum_{\ell=0}^2r_\ell\left(x-\xi\right)^{\ell}$ and $a_1(x)=\sum_{m=-1}^1 s_m\left(x-\xi\right)^m$.

\end{theo}

Specifically, for each of the exceptional Laguerre and Jacobi families, the recursion formulas are provided below in Table \ref{tab-recursion}.  The initial moments with which to begin the recursion are provided in Appendix \ref{App:AppendixA}.

\begin{table}
\begin{tabular}{cc|c|}
	\cline{3-3}
	&& Recursion Formula \\
	\hline
	\multicolumn{1}{ |c  }{\multirow{2}{*}{Exceptional Laguerre} } &
	\multicolumn{1}{ |c| }{Type I } & $\widetilde \mu_{k+2}=(2\alpha+k)\widetilde \mu_{k+1}+\alpha(1-k)\widetilde \mu_k$\\ 
	\cline{2-3}
	\multicolumn{1}{ |c| }{}& \multicolumn{1}{ |c| }{Type II } & $\widetilde \mu_{k+2}=(2\alpha+k)\widetilde \mu_{k+1}+\alpha(1-k)\widetilde \mu_k$  \\ 
	\cline{2-3}
	\multicolumn{1}{ |c| }{}	& \multicolumn{1}{ |c| }{Type III } & $\widetilde\mu_{k+2}=k\widetilde \mu_{k+1}-\alpha(1-k)\widetilde \mu_k$ \\
	\hline
	\multicolumn{1}{ |c| }{Exceptional Jacobi}& & $\widetilde\mu_{k+2}=\left[\frac{(2-\alpha-\beta-2k)\xi+\beta-\alpha}{\alpha+\beta+k}\right]\widetilde \mu_{k+1}+\left[\frac{(k-2)(1-\xi^2)}{\alpha+\beta+k}\right]\widetilde \mu_k$\\
	\hline
\end{tabular}$\quad$

\
	\caption{Recursion formulas for $X_1$ moments.}
	\label{tab-recursion}
\end{table}

\begin{rem*}
The details of the recursion formula of Theorem \ref{theo: Recursion formula} are provided below and the details of the initial moments are given in Appendix \ref{App:AppendixA}.  Before the proof, we provide two critical observations.  First, note that for functions $f$ and $g$, which are smooth on $I$, the associated moment functionals satisfy:
	\[\left<\widehat W', f\right>=-\left<\widehat W,f'\right> \mbox{ and } \left<g\widehat W, f\right>=\left<\widehat W, fg\right>\,,\] where $\left<\cdot, \cdot\right>$ represents the inner product with respect to the Lebesgue measure on $I$.  
	
	Second, the second order linear operator given by \[\ell[y]=a_2y''+a_1y'+a_0y\] may be written as a symmetry equation \[a_2y'+(a_2'-a_1)y=0\] that is solvable by the weight function $\widehat W$ to which the associated eigenfunctions are orthogonal.  
	
	Combining these two notes, we observe that 
	\begin{align*}
	\left<a_2 \widehat W',(x-\xi)^k\right>&=\left<\widehat W',a_2(x-\xi)^k\right>=-\left<\widehat W,\left(a_2(x-\xi)^k\right)'\right>\\
	& =k\left<\widehat W, a_2(x-\xi)^{k-1}\right>-\left<\widehat W,a_2'(x-\xi)^k\right>\,.
	\end{align*}  Therefore, for $k \in \mathbb N$,
	\begin{align}\nonumber
		0 & =\left<a_2\widehat W'+(a_2'-a_1)W, (x-\xi)^k\right>\\ \nonumber
		& =\left<a_2\widehat W', (x-\xi)^k\right>+\left<a_2'\widehat W,(x-\xi)^k\right>-\left<a_1\widehat W, (x-\xi)^k\right>\\ 
	\label{eq: Inner Prod} & =-k\left<\widehat W,a_2(x-\xi)^{k-1}\right>-\left<a_1\widehat W, (x-\xi)^k\right>\,. 
	\end{align}
\end{rem*}

\begin{proof}
	Using equation \eqref{eq: Inner Prod}, we aim to prove a general recursion formula.  First rewrite the coefficients $a_2$ and $a_1$ from the differential expression to be in terms of power so $\xi$.  In the case of the $X_1$-Laguerre families, $a_2$ is of degree 1; for the $X_1$-Jacobi family, $a_2$ is of degree 2.  Therefore, we have $$a_2(x)=\sum\limits_{\ell=0}^2 r_\ell(x-\xi)^\ell,$$ where the coefficients $r_\ell$ are appropriately chosen as indicated below:
	\begin{align}\label{a_2 coefficients}
	 a_2(x)=\begin{cases} 
	-(x-\xi)-\xi & \text{Types I, II, III }X_1\text{-Laguerre}\\
	-(x-\xi)^2-2\xi(x-\xi)+(1-\xi^2)& X_1\text{-Jacobi}\,.
	\end{cases}
	\end{align}
	Since $a_1$ maybe written as $p'+ q+\frac{2b'p}{b}$, using a degree argument, $a_1$ may be written as $$a_1(x)=\sum\limits_{m=-1}^1 s_m(x-\xi)^m,$$ where the coefficients $s_m$ are appropriately chosen as indicated below:
	\footnotesize
	\begin{align}\label{a_1 coefficients} a_1(x)=\begin{cases} 
		(x-\xi)+(1+2\xi)+2\xi(x-\xi)^{-1}& \text{Type I }X_1\text{-Lag.}\\ 
		(x-\xi)-(3+2\xi)+2\xi(x-\xi)^{-1}&\text{Type II }X_1\text{-Lag.}\\ 
		(x-\xi)+2\xi(x-\xi)^{-1}& \text{Type III }X_1\text{-Lag.}\\
		-(\alpha+\beta)(x-\xi)+(2-\alpha-\beta)\xi+\beta\alpha+2(\xi^2-1)(x-\xi)^{-1}&X_1\text{-Jacobi}.
	\end{cases}
	\end{align}	
\normalsize	
	Substituting $a_2$ and $a_1$ into \eqref{eq: Inner Prod}, rearranging, and collecting coefficients produces 
	\[	0=-(kr_0+s_{-1})\widetilde \mu_{k-1}-(kr_1+s_0)\widetilde \mu_k-(kr_2+s_1)\widetilde \mu_{k+1}\quad (\text{for } k \in \mathbb N)\,.\] In other words, for $k \in \mathbb N$, 
	\begin{align}\label{recursion formula} 
		\widetilde \mu_{k+1}=-(kr_2+s)^{-1}\left((kr_0+s_{-1})\widetilde \mu_k-(kr_1+s_0)\widetilde \mu_{k+1}\right)\,.\end{align}  Shifting $k \mapsto k+1$, we obtain the result. \end{proof}

\begin{ex*}
	 It is a short exercise, using $\xi=\alpha$ and equations \eqref{a_2 coefficients} and \eqref{a_1 coefficients}, for the reader to see that \eqref{recursion formula} simplifies to that given in Table \ref{tab-mus}.
	
	We will verify for $n=2$, that the polynomials in Theorem \ref{theo: determinantal} indeed agree with \eqref{eq: TypeIII Lag Poly}. 
	
	For $n=2$, recall that $c_{1,0}=1$ and $c_{1,1}=0$ as in Table 5.  Therefore, \begin{align*}
	L_{1,2}^{III,\alpha}(x)&=\left| \begin{matrix} 0 & 2\alpha & 0\\ \widetilde \mu_0 & \widetilde \mu_1 & \widetilde \mu_2\\ 1 & (x-\alpha) & (x-\alpha)^2\end{matrix}\right|\\
	& =2\alpha \widetilde \mu_2-2\alpha \widetilde \mu_0(x-\alpha)^2.\end{align*}
	
	Using the recursion formula in part (a) of Theorem \ref{theo: Recursion formula} with $k=0$, \[\widetilde \mu_0=\frac{\widetilde \mu_2}{-\alpha}\,,\] we have, up to normalization, 
	\[L_{1,2}^{III,\alpha}(x)=2\widetilde \mu_2\left(x^2-2\alpha x+\alpha (\alpha+1)\right).\]  This is in agreement with the polynomial given in  \eqref{eq: TypeIII Lag Poly} since they both span the same eigenspace.  
\end{ex*}

\begin{rem}
Although different in many ways, the moment representations for the Type I and II Laguerre polynomials \emph{only} differ in the exceptional condition \eqref{Exceptional Condition} (up to normalization, see Table \ref{tab-norms}). The $2$nd to $(n+1)$-st rows in the matrix of Theorem \ref{theo: determinantal}, the moment recursion formulas in Table \ref{tab-recursion} and the initial moments $\widetilde \mu_0$ and $\widetilde \mu_1$ in Table \ref{tab-mus} turned out completely identical for the Type I and II Laguerre polynomials.
\end{rem}

\begin{rem}
	There are additional ways to compute the moments of the $X_1$ polynomial families. In particular, generating functions may be used.  This method may be seen for the Type I Laguerre family in \cite{Liaw-Osborn}.
\end{rem}

\section{Some observations regarding higher co-dimensions}\label{s-Xm}
The introduction of the Darboux approach to the field \cite{Quesne, Quesne2} allowed the community to study $X_m$ orthogonal polynomials when $m>1$. The approach towards Sections \ref{s-prelims} and \ref{Exceptional Condition} was taken from the perspective of allowing higher co-dimensional sequences.

Further, it is not hard to see that Lemma \ref{lemma} can be generalized to the case of $m>1$ by simply replacing the exceptional condition by a set of $m$ exceptional conditions. We must assume that condition \eqref{eq: exceptional condition} holds for $x=\xi_i$ where $\xi_i$ ($i=1,\hdots,m$) denote the $m$ roots of the function $\eta$. Namely, we just replace \eqref{eq: exceptional condition} by the $m$ exceptional conditions
\begin{align}\label{e-mexceptional}
\left[2p\eta'y'-\left(p\eta''+\frac{p'\eta'}{2}-s\eta'\right)y\right]\bigg|_{x=\xi_i}=0
\qquad\text{for }i=1,\hdots,m.
\end{align}

However, the flag and the moment \emph{recursion} formulas occur less straight forward when $m>1$.

\subsection{Possible Flags}
As the codimension increases, there become increasingly more possibilities for the organization of the flag; that is that there are more choices in which degrees are removed from the sequence.  Certain choices, simplify subsequent computations like moment recursion formulas and the set-up of matrix $A$ in Theorem \ref{theo: determinantal}. The importance of making ``good'' choices that simplify the computations already became clear in \cite{Liaw-Osborn}. Here we do not go into the details of how to simplify the computations, but rather provide some preliminary discussions on what kind of choices are allowed in the case of the Type I $X_2$-Laguerre polynomials.

For example, consider the case when $m=2$.
Of course, we must take the first flag element to be 
\[
v_2(x) = L^\alpha_2(-x).
\]
After some consideration, it appears that for $m=2$ the choice
\[
v_3(x)= (x-\xi_1)^2(x-\xi_2+1)
\]
is appropriate. To prove that a set $\{v_2, v_3, \hdots\}$ indeed forms a flag, it suffices to take polynomials that satisfy the exceptional conditions \eqref{e-mexceptional} and are of the appropriate degree.
As a result, there are now several choices for higher degree flag elements. For $n\ge 4$ we set
\[
v_n(x) = (x-\xi_1)^k (x-\xi_2)^{n-k}
\qquad
\text{where we choose } 2\le k \le n-2.
\]


In general, for $m\ge2$, we can take
\begin{align*}
v_m(x) &= L_m^\alpha(-x),\\
v_n(x)&= \prod_{i=1}^m (x-\xi_i)^{k_i}
\qquad\text{for }n\ge 2m, k_i\ge 2\text{ and } \sum k_i = n.
\end{align*}
We expect that the flag elements of degree $m+1,\hdots, 2m-1$ are more complicated. But in principle the only requirements are that they have the correct degree and that they satisfy the exceptional conditions \eqref{e-mexceptional}. 

\subsection{Recursion Type Formulas}
The general method of finding recursion type formulas for moments of the exceptional weights applies to the higher co-dimension setting. In \cite{Liaw-Osborn}, using an adjusted moment greatly simplifies the computations. As in the discussion of possible flags above, there is again much more freedom, and it is not expected that all choices will yield favorable results.

Currently, it is clear that a \emph{good} adjustment for the moments is given by
\[
\widetilde\mu_{(l_1,\hdots,l_m)}
:=
\int_I \prod_{i=1}^m(x-\xi_i)^{l_i} \widehat W(x) dx
\qquad\text{where }l_i\in \N_0.
\]
For example for $m=2$, the adjusted moments of interest will take the form
\[
\widetilde\mu_{(l_1,l_2)}
:=
\int_I (x-\xi_1)^{l_1}(x-\xi_2)^{l_2} \widehat W(x) dx,
\]
so that recursion formulas will also need to generate $\widetilde\mu_{(l_1,l_2)}$ for all $(l_1,l_2)\in \N_0\times \N_0$.

The choices made for the flag and for the ansatz used to generalize \eqref{e-ansatz} will determine which moments will occur in the generalization of matrix $A$. This in turn tells us, which of the adjusted moments we need to generate.
}

\appendix
\section{Initial moments $\widetilde\mu_0$ and $\widetilde\mu_1$} \label{App:AppendixA}
\begin{lem}\label{lem-tableentries}
The initial moments for the $X_1$-exceptional orthogonal polynomial systems of Laguerre and Jacobi are given in Table \ref{tab-mus}, where the Gamma function is given by \[\Gamma(x):=\int_0^\infty t^{x-1}e^{-t}\,dt\] and the incomplete Gamma function by \[\Gamma(a,x):=\int_x^\infty t^{a-1}e^{-t}\,dt \mbox{ for } x>0\,.\]  
Further we used the notation 
$$
J_1 = \frac{-1}{(\alpha+1)(\alpha+\beta)} \; F_1 \left(1,-\beta,1,\alpha+2;-1,\frac{\beta-\alpha}{\alpha+\beta}\right)
$$
as well as
$$
J_2 = \frac{-1}{(\beta+1)(\alpha+\beta)} \; F_1 \left(1,-\alpha,1,\beta+2;-1,\frac{\alpha-\beta}{\alpha+\beta}\right).$$
\end{lem}

\footnotesize
	\begin{table}
		\begin{tabular}{cc|c|c|}
			\cline{3-4}
			&&$\widetilde \mu_0$ & $\widetilde \mu_1$\\
			\hline
			\multicolumn{1}{ |c  }{\multirow{2}{*}{Exceptional Laguerre} } &
			\multicolumn{1}{ |c| }{I } & $\Gamma(\alpha)-2e^\alpha\alpha^\alpha\Gamma(1+\alpha)\Gamma(-\alpha,\alpha)$ & $e^\alpha\alpha^\alpha\Gamma(1+\alpha)\Gamma(-\alpha, \alpha)$\\ 
			\cline{2-4}
			\multicolumn{1}{ |c| }{}& \multicolumn{1}{ |c| }{II } &$\Gamma(\alpha)-2e^\alpha\alpha^\alpha\Gamma(1+\alpha)\Gamma(-\alpha,\alpha)$ & $e^\alpha\alpha^\alpha\Gamma(1+\alpha)\Gamma(-\alpha, \alpha)$\\ 
			\cline{2-4}
			\multicolumn{1}{ |c| }{}	& \multicolumn{1}{ |c| }{III } & $\frac{-\Gamma(\alpha+1)}{\alpha}$ & $e^{-\alpha}(-\alpha)^\alpha \Gamma(1+\alpha)\Gamma(-\alpha, -\alpha)$ \\
			\hline
			\multicolumn{1}{ |c| }{Exceptional Jacobi}& &$\frac{\Gamma(\alpha+1)\Gamma(\beta+1)(\alpha+\beta)}{2\alpha\beta\Gamma(\alpha+\beta+2)} + 
		\frac{(J_1+J_2)(2\alpha\beta-\alpha-\beta)}{\alpha\beta}$ & $\frac{4}{\beta-\alpha}(J_1+J_2)$\\
			\hline
		\end{tabular}
		
		\
		
		\
		
		\caption{The values of the initial moments $\widetilde\mu_0$ and $\widetilde\mu_1$.}
		\label{tab-mus}
	\end{table}
\normalsize

In the case of the Type I Laguerre polynomial family, the proof has been published in \cite[Theorem 4.1]{Liaw-Osborn}.  We will prove that the moments given in Table \ref{tab-mus} are correct for the Type III $X_1$-Laguerre and $X_1$-Jacobi polynomial families.  The proof for the Type II $X_1$-Laguerre family follows in analogy to the other Laguerre types.  

\begin{proof}[Type III $X_1$-Laguerre Proof.]
We prove that the expressions given in Table \ref{tab-mus} are the moments $\widetilde \mu_0$ and $\widetilde \mu_1$ associated with the Type III $X_1$-Laguerre expression.  Set \[E_a(x)=\int_1^\infty e^{-xt} t^{-a}\,dt\,,\quad \mbox{ where } \quad x>0\,.\] Recalling the definitions of $\Gamma(x)$ and $\Gamma(a,x)$, these functions and the exponential integral function, $E_a$, are related via \[E_a(x)=x^{a-1}\Gamma(1-a, x)\,.\]  In addition,
\[
(a-1)E_a(x)=e^{-x}-xE_{a-1}(x)\,.
\]  

Following a chain of manipulations, which may be found in \cite{Liaw-Osborn}, we note the following relation: \[\int_0^\infty \frac{e^{-x}x^\beta}{x-\alpha}dx= e^{-\alpha} E_{1+\beta}(-\alpha)\Gamma(1+\beta)\,,\quad \mbox{ for } \alpha>0\,,\; \beta>-1\,.\]  

In consequence, we obtain the following expression for the first adjusted moment, $\widetilde \mu_1=\int_0^\infty \frac{x^\alpha e^{-x}}{x-\alpha}\,dx$:
\begin{align*}
	\widetilde \mu_1& = e^{-\alpha} E_{1+\alpha}(-\alpha) \Gamma(1+\alpha)\\
	& = e^{-\alpha}(-\alpha)^\alpha \Gamma(-\alpha, -\alpha) \Gamma(1+\alpha)\,.
\end{align*}

Notice that \[\widetilde \mu_2=\int_0^\infty (x-\alpha)^2\frac{x^\alpha e^{-x}}{(x-\alpha)^2}\,dx=\int_0^\infty x^\alpha e^{-x}\,dx=\Gamma(\alpha+1)\,.\] 

To finish the proof, we set $k=0$, use the recursion formula Table \ref{tab-recursion}, and solve for $\widetilde\mu_0$ to show $\widetilde \mu_0=\frac{-\Gamma(\alpha+1)}{\alpha}$.
\end{proof}

\begin{proof}[$X_1$-Jacobi Proof.]
We prove that the expressions given in Table \ref{tab-mus} are the moments $\widetilde \mu_0$ and $\widetilde \mu_1$ associated with the $X_1$-Jacobi expression.  Per \eqref{eq: d-moments}, 
   \[ \widetilde \mu_1:=\int\ci{I} (x-\xi) \widehat W(x) dx. \]

For the $X_1$-Jacobi system, $\xi = \frac{\alpha+\beta}{\beta-\alpha}$, and
   \[\widehat W_1^{\alpha, \beta}(x)=\frac{(1-x)^\alpha (1+x)^\beta}{\left(P_1^{(-\alpha-1,\beta-1)}(x)\right)^2}\,.\]

Substituting these into the equation for $\widetilde \mu_1$, along with the classical Jacobi polynomial in the denominator of the weight function, we find that 
   \[ \widetilde \mu_1=\frac{4}{\beta-\alpha}\int_{-1}^1 \frac{(1-x)^{\alpha}(1+x)^{\beta}}{(\beta-\alpha)x-\alpha-\beta} dx. \]

We recast this integral as the sum of two integrals, so that now
   \[ \widetilde \mu_1=\frac{4}{\beta-\alpha}[J_1+J_2],  \]
where
   \[ J_1 = \int_0^1 \frac{(1-x)^{\alpha}(1+x)^{\beta}} {(\beta-\alpha)x-\alpha-\beta} dx,  \]
and
   \[ J_2 = \int_0^1 \frac{(1-x)^{\beta}(1+x)^{\alpha}} {(\alpha-\beta)x-\alpha-\beta} dx. \]
   
To obtain the value of these integrals, we note that the first Appell hypergeometric series has an integral representation if two of its parameters meet certain restrictions.  Namely, we find
   \[ F_1(a,b_1,b_2,c;x,y) = \frac{\Gamma(c)}{\Gamma(a)\Gamma(c-a)} \int_0^1 t^{a-1}(1-t)^{c-a-1}(1-xt)^{-b_1}(1-yt)^{-b_2} dt \]
for Re $c >$ Re $a > 0,$ see \cite[Chapter 9]{bailey}.

Observe that setting $a = 1, b_1 = -\beta, b_2 = 1, c = \alpha + 2, x = -1,$ and $y = \frac{\beta-\alpha}{\alpha+\beta}$ allows us to equate $J_1$ with the Appell series.  These values also satisfy the restrictions for $a$ and $c$.  Similarly, we set $a = 1, b_1 = -\alpha, b_2 = 1, c = \beta + 2, x = -1,$ and $y = \frac{\alpha-\beta}{\alpha+\beta}$ to equate $J_2$ with the Appell series.

This results in the following values:
\[ J_1 = \frac{-1}{(\alpha+1)(\alpha+\beta)} F_1 \left( 1, -\beta, 1, \alpha+2; -1, \frac{\beta-\alpha}{\alpha+\beta} \right), \]
and
\[ J_2 = \frac{-1}{(\beta+1)(\alpha+\beta)} F_1 \left( 1, -\alpha, 1, \beta+2; -1, \frac{\alpha-\beta}{\alpha+\beta} \right). \]

As mentioned above, $\widetilde \mu_1=\frac{4}{\beta-\alpha}[J_1+J_2]$.

We calculate $\widetilde \mu_0$ indirectly by first finding $\widetilde \mu_2$, an easier computation.  Again, per equation \eqref{eq: d-moments},
   \[ \widetilde \mu_2:=\int\ci{I} (x-\xi)^2 \widehat W(x) dx. \]

Substituting the expressions for $\xi$ and the exceptional weight, we find that
   \[ \widetilde \mu_2=\frac{4}{(\beta-\alpha)^2} \int_{-1}^1 (1-x)^{\alpha}(1+x)^{\beta} dx. \]

A standard table of integrals informs us that
   \[ \int_a^b (x-a)^m (b-x)^n dx = (b-a)^{m+n+1} \; \frac{\Gamma(m+1)\Gamma(n+1)}{\Gamma(m+n+2)}, \]

for $m,n > -1$, and $b > a.$

Setting $a = -1, b = 1, m = \beta,$ and $n = \alpha$ allows us to compute the value of $\widetilde \mu_2$.  The restriction on $m$ and $n$ matches those for the Jacobi parameters $\alpha$ and $\beta$:
   \[ \widetilde \mu_2= \frac {4 \, \Gamma(\alpha+1)\Gamma(\beta+1)} {(\beta-\alpha)^2 \, \Gamma(\alpha+\beta+2)}. \]
   
Finally, using the recursion formula given in Table \ref{tab-recursion} for the $X_1$-Jacobi system, and setting $k=0$, we calculate $\widetilde \mu_0$.  After substituting in the expression for $\xi$ in terms of $\alpha$ and $\beta$, and further simplifications, we obtain the following value:
   \[ \widetilde \mu_0 = \frac {(\alpha+\beta)\Gamma(\alpha+1)\Gamma(\beta+1)} {2\alpha\beta\Gamma(\alpha+\beta+2)} + \frac {(2\alpha\beta-\alpha-\beta)(J_1+J_2)} {\alpha\beta}. \]
This concludes the proof of Lemma \ref{lem-tableentries}.
\end{proof}

\end{document}